f

\documentclass{crmp-l}

\copyrightinfo{2008}{American Mathematical Society}

\usepackage{amssymb}
\usepackage{latexsym} 
\usepackage{euscript}
\usepackage{txfonts}
\usepackage[active]{srcltx}

\newtheorem{theorem}{Theorem} 
\newtheorem{proposition}[theorem]     {Proposition}
\newtheorem{corollary}[theorem] {Corollary}

\theoremstyle{definition}
\newtheorem{example}[theorem]   {Example}

 \def\mylabel#1{\label{#1}}

\newcommand{\RR}{{\bf R}}     
\newcommand{\R}[1]{{\bf R}^{#1}} 
\newcommand{\FF}{{\bf F}}     
\newcommand{\CC}{{\bf C}}     
\newcommand{\mcal}[1]{\mathcal {#1}} 

\newcommand{\msf}[1]{\mathsf {#1}}

\newcommand{\ZZ}{\ensuremath{{\bf Z}}}	   

\newcommand{\NN}{\ensuremath{{\bf N}}}	   
\newcommand{\Np}{\ensuremath{{\bf N}_{+}}}  
\renewcommand{\S}[1]{\ensuremath{{\bf S}^{#1}}} 
\renewcommand{\P}[1]{\ensuremath{{\bf P}^{#1}}} 
\newcommand{\D}[1]{\ensuremath{{\bf D}^{#1}}} 

\newcommand{\Fix}[1]{\ensuremath{\operatorname{\mathsf {Fix}}(#1)}}
\newcommand{\co}{\colon\thinspace} 

\newcommand{\cc}{\ensuremath{\mathfrak c}}

\newcommand{\ee}{\ensuremath{\mathfrak e}}
\newcommand{\ff}{\ensuremath{\mathfrak f}}
\renewcommand{\gg}{\ensuremath{\mathfrak g}}
\newcommand{\hh}{\ensuremath{\mathfrak h}}

\newcommand{\jj}{\ensuremath{\mathfrak j}}
\newcommand{\kk}{\ensuremath{\mathfrak k}}
\newcommand{\mm}{\ensuremath{\mathfrak m}}
\newcommand{\nn}{\ensuremath{\mathfrak n}}
\newcommand{\oo}{\ensuremath{\mathfrak o}}
\newcommand{\pp}{\ensuremath{\mathfrak p}}

\renewcommand{\ss}{\ensuremath{\mathfrak {s}}}
\newcommand{\ttt}{\ensuremath{\mathfrak t}}

\newcommand{\vv}{\ensuremath{\mathfrak v}}
\newcommand{\ww}{\ensuremath{\mathfrak w}}

\newcommand{\zz}{\ensuremath{\mathfrak z}}
\renewcommand{\ll}{\ensuremath{\mathfrak l}} 

\newcommand{\ad}{\ensuremath{{\mathsf {ad}}}}
\newcommand {\w}[1] {\ensuremath{\widetilde {#1}}}
\newcommand{\spec}[1]{\ensuremath{\operatorname{\msf{spec}}\,(#1)}}
\newcommand{\bo}[1]{\ensuremath{\boldsymbol {#1}}}

\newcommand{\lam}{\ensuremath{\lambda}}

\newcommand{\del}{\ensuremath{\delta}}

\newcommand{\Gam}{\ensuremath{\Gamma}}
\newcommand{\gam}{\ensuremath{\gamma}}

\begin{document}

\title[Actions of Lie  groups  and Lie algebras]
{Actions of Lie groups and Lie algebras on manifolds}

\author[M. W. Hirsch]{Morris W. Hirsch}
\address{University of California at Berkeley, 
 University of Wisconsin at Madison}
\curraddr{7926 Hill Point Road, Cross Plains, WI 53528}
\email{mwhirsch@chorus.net}
\thanks{I thank M. Belliart, K. DeKimpe, W. Goldman,  G. Mostow, J. Robbin,
  D. Stowe, F.-J. Turiel and J. Wolf  for invaluable help.}

\date{December 5, 2008} 
\dedicatory{Dedicated to the memory of Raoul Bott}

\begin{abstract} 

Questions of the following sort are addressed:
 Does  a given Lie group  or Lie algebra act
  effectively on a given manifold?
 How smooth can such actions be?   
  What fixed point sets are possible?
What happens under perturbations? 
Old results are summarized, and new ones presented, including:
For every integer $n$ there are solvable (in some cases, nilpotent)
Lie algebras $\gg$ 
that have effective $C^\infty$ actions
on all $n$-manifolds, but on some (in many cases, all)
$n$-manifolds, $\gg$ does not have effective analytic actions.
   
\end{abstract}

\maketitle
  \tableofcontents
\section*{Introduction} 
Lie algebras were  introduced  by Sophus Lie
under the name ``infinitesimal group,'' meaning the  germ of a finite
dimensional, locally transitive Lie algebra of analytic vector fields
in $\R n$.
In his 1880 paper {\em Theorie der Transformationsgruppen}
  \cite{Lie80, Hermann75} and his later book with F. Engel
  \cite{LieEngel93}, Lie classified infinitesimal groups acting in dimensions
  $1$ and $2$ up to analytic coordinate changes.  This  work
  stimulated much research, but attention soon shifted to the
  classification and representation of abstract Lie algebras
  and Lie groups.  Later the topology of Lie groups was studied, with
  fundamental contributions by Bott.

In 1950  G.\,D. Mostow \cite{Mostow50} completed  Lie's
program of classifying effective transitive surface actions.\footnote
 {For each equivalence class of transitive surface actions, Mostow
 describes a representative Lie algebra by formulas for a basis of
 vector fields.  Determining whether one of these representatives is
 isomorphic to a given Lie algebra can be nontrivial.  Here the
 succinct summary of the classification in M. Belliart
 \cite{Belliart97} is useful.}
One of his major results is:
\begin{theorem}	[{\sc Mostow}]  \mylabel{th:mostow}
A surface $M$ without boundary admits a transitive Lie group
 action if and only if $M$ is a plane, sphere, cylinder,  torus,
projective plane, M\"obius strip or 
 Klein bottle.
\end{theorem}
\noindent By a curious coincidence these are the only surfaces
without boundary admitting nontrivial compact Lie group actions (folk
theorem).

The following nontrivial extension of Theorem \ref{th:mostow} deserves
to be better known: 
\begin{theorem}		\mylabel{th:higher}
Let  $G$ be a Lie group and $H$ a closed subgroup such that the
manifold $M=G/H$
is compact.   Then $\chi(M)\ge 0$, and  if $\chi(M) >0$ then 
$M$ has finite fundamental group.
\end{theorem}
\noindent 
This is due to  
 Gorbatsevich  {\em et al.\ }\cite[Corollary 1, p. 174]{Vinberg77}.
 See also 
 Felix {\em et al.\ }\cite[Prop. 32.10]{Halperin01},
Halperin \cite{Halperin77},
Mostow \cite{Mostow05}.

While much is known about the topology of compact group actions, there
has been comparatively little progress on classification of actions of
Lie algebras and noncompact groups, an exception being D. Stowe's
classification \cite{Stowe83} of analytic actions of $SL (2,\RR)$ on
compact surfaces.  The present article addresses the easier tasks of
deciding whether a group or algebra acts nontrivially on a given
manifold, determining the possible smoothness of such actions, and
investigating  their orbit structure.  Most proofs are
omitted or merely outlined, with details to appear elsewhere.

The low state of current  knowledge is illustrated by the lack of both
 counterexamples and  proofs for the following 

\smallskip 
{\bf Conjectures.} Let $\gg$ denote a real, finite dimensional Lie algebra.

\begin{description}

\item[(C1)] {\em If $\gg$ has effective actions on $M^n$, then $\gg$
  also has  smooth effective actions on $M^n$.}

\item[(C2)] 
{\em If $\gg$ is semisimple and has  effective smooth actions on
  $M^n$, $n \ge 2$,
then $\gg$ also has  effective analytic actions on $M^n$.} 

\end{description}
But however  plausible these statements may appear, they can't both be true:
\begin{itemize}
\item {\em {\em (C1)} or {\em (C2)} is false for $\gg=\ss\ll (2,\RR)$.}
\end{itemize}
For $\ss\ll (2,\RR)$ has effective actions on every $M^2$ (Theorem
\ref{th:poly}), but no 
effective analytic action if $M^2$ is compact with Euler
characteristic $\chi (M^2)<0$  (Corollary
\ref{th:smoothanal}(b)).  

It is unknown whether such a surface
can support a smooth effective action  $\beta$ of  $\ss\ll (2,\RR)$.
If it does, 
Theorem \ref{th:rho}(ii) implies that the vector fields $X^\beta$ are
infinitely flat at the fixed points of $\ss\oo (2,\RR)^\beta$.

The analog of (C2) for 
nilpotent  algebras is false.   If $\nn$ denotes  the Lie algebra of
$3\times 3$ niltriangular real matrices, by Theorem \ref{th:st3r}
and Example \ref{th:strn}: 
\begin{itemize}
\item{\em On every connected surface $\nn \times\nn$ has effective
$C^\infty$ actions, but no effective analytic actions.}
\end{itemize}

Further conjectures and  questions are given below. 

\subsection*{Terminology}  $\FF$ stands for the real field $\RR$, or
the complex field $\CC$.  The complex conjugate of $\lam :=a + \imath
b$ is $\bar \lam:=a-\imath b$. 
The sets of integers, positive integers and natural
numbers are $\ZZ$, $\Np=\{1,2,\dots\}$ and $\NN=0\cup\Np$ respectively.
$i, j, k, l, m,  n,r$ denote natural numbers, assumed positive unless
the contrary is indicated.  $\lfloor s \rfloor$ denotes the largest
integer $\le s$. 

$M$ or $M^n$ denotes an $n$-dimensional analytic manifold, perhaps with
boundary; its tangent space at $p$ is $T_pM$.  
$\vv^s (M)$ denotes the
vector space of $C^s$ vector fields on $M$, with the weak $C^s$
topology ($1\le s\le \infty$).  The Lie bracket makes $\vv^\infty$
a Lie algebra, with analytic vector fields forming a subalgebra.   
The value of $Y\in \vv^1 (M)$ at $p\in M$ is $Y_p$.  The derivative
of $Y$ at $p$ is a linear operator on   $dY_p$ on  $T_pM$.   

Except as otherwise indicated, manifolds, Lie groups and Lie algebras
are real and finite dimensional; manifolds and Lie groups are
connected; and maps between manifolds are $C^\infty$.  
 
$G$ denotes a Lie group with Lie algebra $\gg$ and universal covering
group $\w G$.  The $k$-fold direct product $G\times\dots\times G$ is
$G^k$ and similarly for $\gg$.  
$SL(m, \FF)$ is
the group of $m\times m$ matrices over $\FF$ of determinant $1$, and
$ST (m,\FF)$ is the subgroup of upper triangular matrices.  The
corresponding identity components and Lie algebras are denoted by
$SL_\circ (m,\FF)$, $\ss\ttt (m,\FF)$ and so forth.

An {\em action} $\alpha$ of $G$ on  $M$, 
indicated by $(\alpha, G,M)$,
is a homomorphism $g\mapsto g^\alpha$ from $G$ to the group of
homeomorphisms of $M$ with a continuous {\em evaluation map}
$\mathsf{ev}_\alpha \co G\times M\to M, \,(g,x)\mapsto g^\alpha(x)$.
We call
$\alpha$ {\em smooth},  or {\em analytic}, when
$\mathsf{ev}_\alpha$ has the corresponding property.\footnote
{Most of the results here can be adapted to  $C^r$ actions 
  and local actions}

Small gothic letters denote linear subspaces of
Lie algebras, with 
$\gg$ and  $\hh$ reserved for Lie algebras. 
Recursively define $\gg^{(0)}=\gg$ and $\gg^{(j+1)}={\gg^{(j)}}\,' =
 [\gg^{(j)},\gg^{(j)}]$ = commutator ideal of $\gg^{(j)}$.  Recall
 that $\gg$ (and also $G$) is  {\em solvable} of {\em derived length} $l=\ell
 (\gg)=\ell (G)$ if $l\in\Np$ is the smallest number satisfying
 $\gg^{(l)}=0$.  For example, $\ell (\ss\ttt (m,\FF)) = m$.

$\gg$ is  {\em nilpotent} if there exists $k\in\NN$ such that
$\gg_{(k)}=\{0\}$, where $\gg_{(0)}=\gg$ and
$\gg_{(j+1)}:=[\gg,\gg_{(j)}]$.   It is  known that  $\gg$ is solvable
if and only $\gg'$ is nilpotent. 

$\gg$ is {\em supersoluble} if the spectrum of $\ad\, X$ is real for
all $X\in\gg$, where $\ad:=\ad_\gg$ denotes the adjoint representation
 of $\gg$ on itself defined by $(\ad\, X)Y=[X,Y]$.
Equivalently: $\gg$ is solvable and faithfully represented by upper
triangular real matrices.

An {\em action} $\beta$ of $\gg$ on $M$, recorded as $(\beta, \gg,
M)$, is a continuous homomorphism $X\mapsto X^\beta$ from $\gg$ to
$\vv^\infty (M)$.  An {\em $n$-action} means an action on an
$n$-dimensional manifold.

A smooth  action  $(\alpha, G, M)$ determines a smooth action
$(\hat \alpha,\gg,M)$.
Conversely, if $G$ is simply connected and $(\beta, \gg, M)$ is such
that each  vector field $X^\beta$ is complete (as when  $M$
is compact), then there exists $(\alpha, G, M)$ such that $\beta=\hat
\alpha$.

The {\em orbit} of $p\in M$ under $(\alpha,G,M)$ is $\{g^\alpha (p)\co
g\in G\}$, and the orbit of $p$ under a Lie algebra action
$(\beta,\gg,M)$ is the union over $X\in \gg$ of the integral curves of
$p$ for $X^\beta$.  An action is {\em transitive} if
it has only one orbit.

The  {\em fixed point set} of $(\alpha, G, M)$ is
\[  \Fix \alpha =\{ x\in M\co g^\alpha (x)=x, \ (g\in G)\},\]
  and that of $(\beta,\gg,M)$ is
\[ 
 \Fix \beta :=\{p\in M\co X^\beta_p=0, \ (X\in \gg)\}
\]
The {\em support} of any action $\gamma$ on $M$ is the closure of
$M\,\verb=\=\,\Fix \gamma$.   

An action  is {\em effective} if its kernel is trivial, and {\em
nondegenerate} if the fixed point set of every nontrivial element has
empty interior.  Effective analytic actions are nondegenerate.
 A group action is {\em almost effective} if its
kernel is discrete.  

\section*{Construction of actions}  
Every $G$ acts effectively and analytically on itself
by translation.  Every  $\gg$ admits a faithful
finite dimensional representation $\msf R\co\gg\to \gg\ll (n, \RR)$ by
Ado's theorem (Jacobson \cite {Jacobson62}).  If
$\msf R (\gg)$ has trivial center, it induces  effective analytic
action by  $\gg$ on the projective space $\P {n-1}$ and the sphere
$\S{n-1}$.    

An action gives rise to actions on other manifolds by blowing up
invariant submanifolds in various ways; this preserves effectiveness
and analyticity.  Blowing up fixed points of standard actions of
$ST_\circ (3, \RR)$ on $\P 2, \S 2$ and $\D 2$ yields:

\begin{theorem}		\mylabel{th:st3r}
$ST_\circ  (3, \RR) $  has effective analytic actions on  all
 compact surfaces.\\
  \hspace*{\fill}{\small \em (F. Turiel \cite{Turiel03})}
\end{theorem}
{\bf Conjecture. } {\em $ST_\circ (3, \RR) $ has effective analytic
actions on all surfaces.}

\medskip
Analytic approximation theory is used to prove: 
\begin{theorem}           \mylabel{th:hr}
The vector group $\R m$ has effective analytic actions on $M^n$ if $m\ge
 1,\ n\ge 2$.
\hspace{\fill}{\small\em(M. Hirsch \& J. Robbin
  \cite{HirschRobbin03})} 
\end{theorem}

On open manifolds it is comparatively easy to produce effective Lie
algebra actions:
\begin{theorem}		\mylabel{th:noncompact}
Assume there is an effective action $(\alpha,\gg, W^n)$.  Then a
noncompact $M^n$ admits an effective action $(\beta, \gg, M^n)$ in the
following cases:
\begin{description}

\item[(a)] $M^n$ is parallelizable

\item[(b)] $n=2$  and $W^2$ is nonorientable. 

\end{description}
Moreover $\beta$ can be chosen 
nondegenerate, analytic, transitive or
fixed-point free provided $\alpha$ has the same property. 

\end{theorem}
\begin{proof}
Define $\beta$ as the pullback of $\alpha$ through an immersion
$M^n\to W^n$ (for immersion theory see Hirsch \cite{Hirsch61}, Poenaru
\cite{Poenaru62}, Adachi \cite{Adachi84}).
\end{proof}

\begin{corollary}		\mylabel{th:noncompactcor}
Every noncompact $M^2$ supports effective analytic actions by $\ss\ll
(3,\RR)$ and $\ss\ll (2,\CC)$.  Every parallelizable noncompact $M^n$
has effective analytic actions by $\ss\ll (n+1, \RR)$, by $\ss\ll
(\frac{n}{2}, \CC)$ if $n$ is even, and by $\ss\ll
(\lfloor\frac{n}{2}\rfloor +1, \CC)$ if $n$ is odd.
\end{corollary}

Actions of $G$ on the circle $\S 1$ lift to actions of $\w G$ on
$\RR$, and by compactification to actions on $[0, 1]$.  Such actions
can be concatenated to get effective actions of $\w {G_1}\times\dots
\times \w{G_m}$ on $[0,1]$.  Further topological constructions lead to
effective actions on closed $n$-disks, trivial on the boundary.
Embedding such disks disjointly into an $n$-manifold leads to:
\begin{theorem}           \mylabel{th:poly}
$\w {SL}_\circ (2, \RR)^j \times ST_\circ  (2,\RR)^k\times \R m$ acts
 effectively  on every manifold of
 positive dimension ($j, k, m \ge 0$).
\end{theorem}
\noindent In many cases such actions cannot be analytic and
their smoothness is unknown; but
see Theorem \ref{th:ACcor}. 

\section*{Algebraically contractible groups} 
The actions constructed  above are either analytic or merely
continuous.  Next we exhibit  a large class of solvable groups
having effective actions--- often smooth---  on manifolds of
moderately low dimensions.   In many case these are smooth 
but cannot be analytic.

Let $\mcal E (G)$ denote the space of endomorphisms of $G$,
topologized as a subset of the continuous maps $G\to G$.  We call $G$
and $\gg$ {\em algebraically contractible} (AC) if there is a path 
$\phi=\{\phi_t\}$ in $\mcal E (G)$ joining the 
the identity
endomorphism $\phi_0$ of $G$ to  the trivial  endomorphism $\phi_1$.
Equivalently: $G$ is solvable and simply connected, and the identity
and trivial endomorphisms of $\gg$ are joined by a path
$\psi=\{\psi_t\}$ in the affine variety $\mcal E (\gg)$ of Lie algebra
endomorphisms of $\gg$.  Every path $\psi$ comes from a unique path
$\phi$.

The class of AC groups contains the vector group $\R n$,  the matrix
groups $ ST_\circ (n,\RR)$, \ $\w{ST_\circ }(n,\CC)$, and many of
their subgroups and quotient groups.  It is closed under direct
products.  If $\gg$ is AC and an ideal $\hh$ is mapped into itself
by every endomorphism of $\gg$, then $\hh$ and $\gg/\hh$ are AC.

However, some nilpotent Lie algebras are not AC (DeKimpe
\cite{DeKimpe06}): The derivation algebra of an AC Lie algebra cannot
be unipotent, but there are $8$-dimensional nilpotent Lie algebras
having unipotent derivation algebras (Dixmier \& Lister
\cite{DixmierLister57}, Ancochea \& Campoamor
\cite{AncocheaCampoamor01}).

\begin{proposition}		\mylabel{th:AC}
Assume $G$ is algebraically contractible and $(\alpha,G, M)$
is almost effective.  There is an effective action $(\beta, G, M\times
\RR)$ with the following properties:
\begin{description}
\item[(a)]     $g^\beta (x,0)= (g^\alpha (x), 0)$

\item[(b)]   $g^\beta (x,t) = (x,t)$ if $\vert t\vert \ge 1$.

\item[(c)]   If $\alpha$ is smooth so is $\beta$.
\end{description}
\end{proposition}
\begin{proof}   We can choose the path $\psi\co[0,1]\to\mcal E (\gg)$  in
  the definition of AC to be $C^\infty$ and constant in a neighborhood
  of $\{0,1\}$.  The corresponding path $\phi\co[0,1]\to \mcal E(G)$
  has the same properties.  Extend $\phi$ over $\RR$ by setting
  $\phi_t = \phi_1$ (= the trivial endomorphism) for $t \ge 1$, and
  $\phi_t= \phi_{-t}$ for $t \le 0$.   Now define $\beta$ by
\[
  g^\beta (x, t):= \phi_t (g)^\alpha(x),  \qquad (g\in G, \,(x,t)\in
  M\times\RR). 
\]
\end{proof}

\begin{theorem}		\mylabel{th:ACcor}
Assume $G_i$ is AC and $(\alpha_i, G, \S{n-1})$ is almost effective,
($i=1,\dots,k$).  For every $M^n$ there exists an effective action
$(\delta, G_1\times \dots\times G_k, M^n)$ that is smooth provided the
$\alpha_i$ are smooth.
\end{theorem}
\begin{proof} Let $(\beta_i, G_i, \S{n-1}\times \RR)$
obtained from $\alpha_i$ as in Proposition \ref{th:AC}.  
Through an
identification   $\S{n-1}\times \RR=\D n\,\verb=\=\, (\S{n-1}\cup0)$,
extend $\beta_i$ to an action $(\gam_i,  G_i, \D n)$ with compact
support in $\D n\,\verb=\=\, \S{n-1}$.   (Here $\D n$ is the unit
$n$-disk with boundary $\S {n-1}$.)  Transfer the $\gamma_i$  to
actions $\del_i$ in $k$ 
disjoint coordinate disks $D^n_i \subset M^n$. Define
$\del$ to coincide with  $\del_i$ in $D^n_i$ and to be trivial outside
$\cup_iD^n_i$.
\end{proof}

\begin{corollary}		\mylabel{th:ACcorfaith}
Assume $G_i\subset GL (n,\RR)$ is algebraically contractible and
contains no scalar multiple of the identity matrix, ($i=1,\dots,k$).
Then $G_1\times\dots\times G_k$ has effective smooth actions on all
$n$-manifolds.
\end{corollary}
\begin{proof} The natural actions of $G_i$ on  $\P{n-1}$ and 
  $\S{n-1}$ are smooth and effective.  Apply Theorems \ref{th:ACcor} and 
  \ref{th:noncompact}.
\end{proof}

\section* {The Epstein-Thurston theorem}
D.B.A. Epstein and W.P. Thurston \cite[Theorem 1.1] {ET79} discovered
fundamental lower bounds on the dimensions in which solvable Lie
algebras can act effectively:
\begin{theorem}[{\sc Epstein-Thurston}]   \mylabel{th:ET}
Assume $\gg$ is solvable and has an effective $n$-action.   Then
 $n\ge \ell (\gg) -1$, and  $n\ge \ell (\gg)$ if $\gg$ is nilpotent.
\end{theorem}

 In the critical  dimensions there is further information
on orbit structure:

\begin{theorem}		\mylabel{th:ETcor}
Let $\alpha$ be an effective $n$-action of a solvable Lie algebra
$\gg$.  Assume $n=\ell (\gg)-1$, or $\gg$ is nilpotent and $n=\ell
  (\gg)$.
\begin{description} 
\item[(i)] 
There is an open orbit.  If
$\alpha$ is nondegenerate the union of the open orbits is dense.

\item[(ii)]   
Assume $\gg^{(n-1)} \subset \cc =$ the center  of $\gg$.  Then: 

\begin{description}

\item[(a)] each nontrivial orbit of 
  $\gg^{(n-1)}$ lies in an open orbit of $\gg$ and has dimension
  $1$,

\item[(b)]  the number of open orbits is
   $\ge \dim  \gg^{(n-1)}$ 

\item[(c)]  if $\alpha$ is nondegenerate then
 $\dim \cc=1$.
\end{description}
\end{description}

\end{theorem}
\begin{proof} 
The union of  orbits of dimensions $<n$ is a closed
set $L$ in which  $\gg^{(\ell (\gg) -1)}$ acts trivially by
Epstein-Thurston.   Therefore $M^n\verb=\=L$, the union of the  open
orbits, is  nonempty because $\alpha$  is  effective, and dense
if $\alpha$ is nondegenerate.  This proves (i).  Next we prove (ii).

\

{\em (a)} 
Let $L$ be a nontrivial orbit of $\gg^{(n-1)}$ and let $O$ be the
orbit of $\gg$ containing $L$.  Then $O$ is an open set because $\dim
(O) =n$ by Epstein-Thurston.  This proves the first assertion of (a). 
To prove the second we can assume the action is transitive. 
Fix a $1$-dimensional subspace
$\zz\subset \cc$ having a   
$1$-dimensional orbit $L_1\subset L$.  
After replacing $O$ by a suitably small open subset, we can assume the
domain of the action is $O=\R {n-1}\times\RR$ with the slices $x\times
\RR$ being the orbits of $\zz$.  The induced action of $\gg$ on the
$n$-dimensional space of $\zz$-orbits kills $\gg^{(n-1)}$ by
Epstein-Thurston.
This implies $L_1=L$, which implies (a).  

\

{\em (b)} 
Suppose 
$\dim \gg^{(n-1)} = s \ge 1 $ and there are exactly $r$ open orbits
$O_i$, $i=1,\dots r$.  
As $\gg$ acts  transitively in  $O_i$  and $\gg^{(n-1)}$ is
central, there is a  codimension-one subalgebra $\kk_i\subseteq
\gg^{(n-1)}$  acting trivially in $O_i$.  
If $1\le r <s$ then $\cap_i\kk_i$ has positive dimension and  acts
trivially in each open 
orbit, and also in  all other orbits by
Epstein-Thurston.  This implies (b).

\

{\em (c)} Assume $\alpha$ is nondegenerate. By (a) there is an open
orbit $O$, which we can assume is the only orbit.  Let $O$, $L$, $\zz$
be as in the proof of (a).  If (c) is false we choose $\zz$ so that
the central ideal $\jj :=\gg^{(n-1)} + \zz $ has dimension $\ge 2$.
In the proof of (a) we saw that every nontrivial orbit of $\jj$ is
$1$-dimensional, hence every  orbit of $\jj$ is $1$-dimensional
because $\alpha$ is transitive and $\jj$ is central.
Therefore for every $p\in O$ there is a maximal nontrivial linear
subspace $\kk_p \subset \jj$ annihilated by $\alpha$.  As $\alpha$ is
transitive and $\jj$ is central, all the $\kk_p$ coincides with an ideal
that acts trivially in $O$.  This contradicts the
assumption that $\alpha$ is nondegenerate.
\end{proof}

\begin{example}		\mylabel{th:strn}
The nilpotent algebra $\nn=\ss\ttt (n+1,\RR)'\times \RR$ has derived
length $n$ and $2$-dimensional center $\ss\ttt (n+1,\RR)^{(n-1)}\times\RR$.  Being
algebraically contractible, $\nn$  acts effectively on all $n$-manifolds
by Corollary \ref{th:ACcorfaith}.  On the other hand, 
Theorem \ref{th:ETcor} implies:
\begin{itemize}
\item {\em   Every $n$-action of $\nn$ is   degenerate and hence nonanalytic.}
\end{itemize}
\end{example}
\section*{Weight spaces and spectral rank}
Let $T\co\gg\to\gg$ be  linear.  For
$\lam$ in the spectrum $\spec T\subset \CC$  define 
the  (generalized) {\em weight space} $\ww(T,\lam)\subset\gg$ to
be  the largest 
$T$-invariant subspace  in which $T$ has spectrum $\{\lam,
\bar\lam\}$.  The largest subspace of $\ww(T,\lam)$ in which $T$ acts
semisimply is

\[
 \mm(T,\lam):= \begin{cases}
 &\text{kernel of $T-\lam I$ \hspace{1in}  \quad if $\lam\in \RR$}\\
 &\mbox{kernel of $T^2- 2(\Re\ee\,\lam) T + \vert\lam\vert^2 I$  \qquad if
                                 $\lam \notin \RR$}
	\end{cases}
\]

For any set $S\subset \CC$ let $\Gam (S)$ denote the additive free abelian
subgroup of $\CC$ generated by $S$.  The rank of 
 $\Gam(\spec T)$ is the
{\em spectral rank}
$\bo r (T)$.
The  rank of
$\Gam(\spec T\,\verb=\=\,\RR)$ is the {\em nonreal spectral rank} $\bo
r_{\msf {NR}} (T)$.   
For a Lie algebra $\gg$ define
\[
  {\bo r} (\gg)=\max_{X\in\gg}{\bo r} (\ad\,X), \quad
  {\bo r}_{\msf{NR}} (\gg)=\max_{X\in\gg}{\bo r}_{\msf{NR}} (\ad\,X)
\]
For example, if  $X \in
\ss\ttt (m, \RR)$ is a sufficiently irrational diagonal matrix then
\[\begin{split}
{\bo r}(\ss\ttt (m, \RR))   &= {\bo r} (\ad\, X)  \ \ \ \ =  m-1,\\
  {\bo r}(\ss\ttt (m, \CC)) &= {\bo r}(\ad\,\, \imath  X) \ \ \,=  m-1,\\
{\bo r}_{\msf{NR}} (\ss\ttt (m, \CC)) &=
{\bo r}_{\msf{NR}}(\ad\,\, \imath X) =  m-1.
  \end{split}
\]
If $\ss$ is semisimple of rank $r$ with  
a Cartan decomposition  $\kk+\pp$,  almost every 
$X$ in the Cartan subalgebra $\kk$ satisfies
 \[{\bo r} (\ss)={\bo r} (\ad_\ss \,X)=r,\qquad  {\bo r}_{\msf{NR}}(\ss)={\bo
   r}_{\msf{NR}} (\ad_\ss \,X)=r 
\] 
(see  Helgason \cite[Prop.~III.7.4] {Helgason62}).

$Y\in \vv^\infty (M)$ is {\em flat at $p\in M$} when its Taylor series
vanishes in local coordinates centered at $p$.    If such a $Y$ is
analytic it is trivial.  
Given
$(\alpha, \gg, M)$ and $p\in M$,  define $\ff_p (\alpha) \subset \gg$
as the set of $Y\in\gg$ such that $Y^\alpha$ is flat at $p$.   This is
an ideal. 
\begin{proposition}		\mylabel{th:rhoprop}
Assume $(\alpha, \gg, M^n)$ is smooth, $X\in \gg$ and $p\in
\Fix{X^\alpha}$.  Suppose  
$\mm(\ad\, X, \lam)\cap\ff_p (\alpha)=0$ for all 
$\lam\in\spec{\ad\,X}\verb=\=\, 0$.
Then
\[ 
\spec{\ad\, X}\subset \Gam   (\spec{dX^\alpha_p})
\] 
and therefore 
\[ n\ge \max \{\bo r (X), \ 2{\bo r}_{\msf{NR}}
 (\ad\,X)\}.
\]
\end{proposition}
\begin{proof}
We can assume $M^n=\R n, \ p=0$.  Write every
$Z\in \vv^\infty (\R n)$ as the formal sum $\sum_{r\in \NN}Z_{(r)}$
where the components of the vector field $Z_{(r)}$ are homogeneous
polynomial functions of degree $r$.  Then $X^\alpha_{(0)}=0$,
\ $X^\alpha_{(1)}= dX^\alpha_p$.  
The {\em order} of $Z$ is the
smallest $r$ for which $Z_{(r)}\ne 0$ if $Z$ is not flat at $0$,
otherwise the order is $\infty$.  
Suppose $Y\in {\kk} (\ad\, X, \lam)$ is not flat at $0$ and has 
finite order $r$.  Then
$(\ad_{\CC\otimes\gg} X - \lam I) Y=0$, implying 
$ [X^\alpha_{(1)}, Y^\alpha_{(r)}] = \lam Y^\alpha_{(r)}$.
Hence 
$  \lam \in \spec {\ad_{\vv^\infty (\R n)}\,  dX^\alpha_p}$.
A calculation shows that $ \spec {\ad_{\vv^\infty (\R n)}\, Z} \subset
\Gam (\spec Z)$ for every linear vector field $Z\co \R n\to\R n$.
Apply this to $Z:=dX^\alpha_p$.  
\end{proof}

\smallskip
The following result is derived from Proposition \ref{th:rhoprop}:
\begin{theorem}	 \mylabel {th:rho}
Suppose  
$(\alpha, \gg, M^n)$ is smooth, $X\in \gg$ and 
 $p\in \Fix{X^\alpha}$. 
\begin{description}
\item[(i)] 
 Assume
 ${\bo r} (\ad\,X) =n+k >n$.  Then $\ad\, X$ has $k$
  different eigenvalues $\lam\ne 0$ such that
$ \ww (\ad\, X,\, \lam) \subset \ff_p (\alpha)$. 
\smallskip
\item[(ii)] \ Assume $2{\bo r}_{\msf{NR}} (\ad\,X)=n$,\ $\alpha$ is
effective, and $\mm(\ad\, X, \lam)\cap\ff_p (\alpha)=0$ for all
$\lam\in\spec{\ad\, X}\,\verb=\=\,\RR$.
  Then $dX^\alpha_p$ has only
nonreal eigenvalues, $X^\alpha$ has index $1$ at $p$, and if $M^n$ is
compact then $\chi (M^n) >0$.
\end{description}
\end{theorem}

\noindent
This has powerful consequences for analytic actions:
\begin{corollary}		\mylabel{th:smoothanal}
Assume 
$(\alpha, \gg, M^n)$ is  effective and analytic and  $X\in \gg$.
\begin{description}

\item[(a)] If $\Fix
{X^\alpha} \ne\varnothing$ then 
$n \ge \max\{\bo r (\ad\,X), 2\bo r_{\msf{NR}} (\ad\,X) \}$.  
\smallskip

\item[(b)] Suppose $M^n$ is compact and $n=2\bo r_{\msf{NR}}
(\ad\,X)$.  Then
 \[  
   \chi (M^n) =\#\Fix {X^\alpha} \ge \#\Fix\alpha.
\]
Therefore $\chi (M^n)\ge 0$, and $\Fix\alpha =\varnothing$ if   $\chi
(M^n)= 0$.  
\end{description}
\end{corollary}
\noindent For surface actions, (b) is due to Turiel
\cite{Turiel03}.

\begin{corollary}		\mylabel{th:estker}
Assume $M^n$ is compact and $\chi (M^n)\ne 0$.  If  $(\alpha, \gg,
M^n)$ is analytic with kernel $\kk$,   then 
$  \dim\kk\ge \max\left\{{\bo r} (\gg)-n, \ {\bo r}_{\msf{NR}}
(\gg) -\left\lfloor \frac{n}{2}\right\rfloor\right\}$.
\end{corollary}

\begin{example}		\mylabel{th:ss}
Assume  $\ss$ is 
semisimple of rank $r$ with a Cartan decomposition $\kk + \pp$
where $\kk$ is a Cartan subalgebra.  The set
$
 U:=\{X\in\kk: \bo r_{\msf{NR}} (\ad\,X) = r\}
$
is dense and open in $\kk$.  Let $(\alpha, \ss, M^n)$ be effective and
analytic, with  $\Fix {X^\alpha} \ne \varnothing$ for some $X\in
U$.  Then Corollary \ref{th:smoothanal} implies:
\begin{itemize}   

\item{\em    $n\ge 2r$.  If $n=2r$ and  $M^n$ is compact then  $\chi (M^n)
=\#\Fix{Y^\alpha}>0$ for all $Y\in \kk$.}

\end{itemize}

\end{example}

\begin{example}		\mylabel{th:exA}
Assume  $m,n,k\in \Np$ with $m\le n$.  
Theorem \ref{th:ACcor}  shows that 
every $n$-manifold supports  a smooth effective action
of $\ss\ttt (m+1,\RR)^k$.  
Because ${\bo r} (\ss\ttt (m+1,\RR)^k)=mk$,    Corollary \ref{th:estker} implies:

\begin{itemize}
 \item {\em Assume $M^n$ is compact and $\chi (M^n)\ne 0$.  If
$(\alpha,\ss\ttt (m+1,\RR)^k, M^n) $ is analytic and effective then
$k\le \left\lfloor\frac{n}{m}\right\rfloor$.}
\end{itemize}
To take a specific example: 
\begin{itemize}
 \item {\em $\ss\ttt (n+1, \RR)\times \ss\ttt (n+1, \RR)$ does not
   have an effective analytic 
   action on any compact $n$-manifold.}
\end{itemize}

\end{example}

\section*{Fixed points} 
For actions of $G$ on compact surfaces $M^2$ the following results  are
known:
\begin{proposition}		\mylabel{th:fixedpts}
{~}
\begin{description}
\item[(a)]  $ST_\circ (2,\RR)$ has effective, fixed-point free  $C^\infty$
actions  on all compact\\ surfaces.  
\hspace*{\fill} {\small \em (Lima \cite{Lima64},
  Plante \cite{Plante86},   Belliart  \&  Liousse \cite {BL96},
Turiel \cite{Turiel89,Turiel06})} 
 \smallskip
\item[(b)] If $G$ acts without fixed point and $\chi (M^2)<0$ then
$ST_\circ (2, \RR)$ is a\\
 quotient group of $G$. \hspace{\fill}{\small \em (Belliart \cite{Belliart97})}
\smallskip
\item[(c)]   If $G$ acts analytically without fixed point, $\chi (M^2)\ge 0$.
\hspace*{\fill}  {\small\em  (Turiel \cite{Turiel03})}
\smallskip
\item[(d)] If $G$ is nilpotent and acts without fixed point, 
  $\chi (M^2)=0$.\\ 
 \hspace*{\fill}{\small\em (Lima \cite{Lima64}, Plante \cite{Plante86})}  
\smallskip
\item[(e)] If $G$  is supersoluble and acts analytically without 
 fixed point,     $\chi (M^2)=0$. 
    \hspace*{\fill}{\small\em  (Hirsch  \&  Weinstein \cite{HW00})}

\end{description}
\end{proposition}
Careful use of the blowup construction shows that some supersoluble
groups have effective analytic surface actions with arbitrarily large
numbers of fixed points:
\begin{theorem}         \mylabel{th:fpts}
Let $M^2_g$ denote a closed surface of genus $g\ge 0$.  For every
$k\in\NN$ there is an 
effective analytic action $(\beta, ST_\circ  (3, \RR), M^2_g)$ 
such that 
\[\#\Fix{\beta}=
        \begin{cases}
            2(g+k+1)    & \text{if $M^2_g$ is orientable,}\\ 
	     g+k     & \text{if $M^2_g$ is nonorientable and $g\ge 1$.}
        \end{cases}
\]
\end{theorem}
\noindent 
On the other hand:
\begin{itemize}
\item{\em Suppose $G$ is not supersoluble.  If
$M^2$ is compact and $(\alpha,G, M^2)$ is  effective and
analytic, then $0\le  \# \Fix\alpha \le \chi (M^2) \le 2$.}
\end{itemize}
This follow from Corollary \ref{th:smoothanal}(b), because  $G$ is
  not supersoluble if and only if ${\bo r}_{\msf {NR}} (G) \ge 1$ .  

\smallskip
{\bf Questions.} Is the analog of Proposition \ref{th:fixedpts}(a)
true for $ST_\circ (3, \RR)$?  Does this group have an effective
analytic action with a unique fixed point on some orientable closed
surface?  Can $ST (3,\RR)$ act effectively on $\S 2$ with a unique fixed point?
Can a smooth effective action of $SL (2,\RR)$ on $\S 1\times\S 1$ have
a fixed point? 

\smallskip
For noncompact group actions in higher dimensions the following are known:

\begin{itemize}
\item {\em $\RR$ acts effectively without fixed point on a compact
 $M^n$ 
 $\Longleftrightarrow \chi (M^n)= 0$.}\\
\hspace*{\fill}{\small(Poincar\'e \cite{Poincare85},
Hopf \cite {Hopf27})}
\smallskip
\item  {\em An algebraic action of a solvable complex algebraic group on
 a complete complex algebraic variety has a fixed
 point.} \hspace{\fill}{\small(Borel \cite{Borel56})}
\smallskip
\item  {\em If $M^n$ is compact, $n=3$ or $4$, and $\chi \left(M^n\right)\ne
 0$,  then every analytic action  of $\R 2$ on $M^n$ has a  fixed point.}   
   \hspace{\fill}{\small(Bonatti \cite{Bonatti92})}

\end{itemize}

\section*{Spectral rigidity}
$\mcal A_1 (\gg, M)$ denotes the space of $C^\infty$ actions of $\gg$ on $M$
under the the smallest topology making the maps the map $\mcal A_1
(\gg, M)\to \vv^1(M), \ \alpha \mapsto X^\alpha$, continuous for all
$X\in \gg$.  An action $(\alpha, \gg, M)$ is {\em spectrally rigid
at} $(X,p)$ if $X\in \gg$, $p\in \Fix {X^\alpha}$, and there exist
arbitrarily small neighborhoods
 $\mcal N\subset \mcal A_1 (\gg, M^n)$ of $\alpha$ and    $ W\subset
M$ of $p$  such that for all $\beta\in \mcal N$: 
\begin{description}

\item[(SR1)]   $\Fix{X^\beta}\cap W\ne\varnothing$

\smallskip
\item[(SR2)]   $q\in \Fix{X^\beta}\cap W \implies
    dX^\beta_q$ and  $dX^\alpha_p$ have the same nonzero eigenvalues.

\end{description}

While spectral rigidity is impossible for nontrivial abelian algebras
and dubious for nilpotent algebras, many solvable and semisimple
algebras exhibit it:

\begin{theorem}		\mylabel{th:SR}
Assume  $(\alpha,\gg,M^n)$ is  effective and
analytic, $X\in \gg$ and ${\bo r}(\ad\,X)=n$.  
Then  $\alpha$ is 
spectrally rigid at  
$(X,p)$ for all $p\in\Fix {X^\alpha}$. 
\end{theorem}
\noindent The proof is based on Proposition \ref{th:rhoprop}.

\medskip

{\bf Conjecture.}  
{\em  An analytic action $\alpha$  of a semisimple Lie
  algebra $\ss$ is spectrally rigid at $(X,p)$ for all $X\in
\ss$,  $p\in\Fix\alpha$.}

\bibliographystyle{amsplain}

\end{document}